\documentclass[twoside,a4paper,12pt]{amsart}
\usepackage{amsfonts, amsbsy, amsmath, amsthm, amssymb, latexsym}
\usepackage{mathrsfs, mathabx}
\usepackage[top=34mm,right=20mm,bottom=34mm,left=20mm]{geometry}
\usepackage{bm}
\usepackage{fancyhdr}
\usepackage{enumerate}
\usepackage{tikz, tikz-cd, tkz-euclide}
\usetikzlibrary{decorations.markings,positioning}
\usepackage{xcolor}
\usepackage{faktor}

\usepackage{hyperref}

\usepackage[official]{eurosym}
\usepackage{array,float,graphicx,epic,eepic}
\usepackage{amssymb, amsmath, amsthm}
\usepackage{changepage}
\usepackage{comment}
\usepackage{thmtools,thm-restate}

\usepackage{url}

\def \G {\Gamma}

\renewcommand{\a}{\alpha}

\renewcommand{\i}{\iota}

\newcommand{\R}{\mathbb{R}}
\newcommand{\Z}{\mathbb{Z}}

 \newcommand{\e}{\varepsilon}
 
\renewcommand{\l}{\lambda} 
 
 \renewcommand{\to}{\rightarrow}
\newcommand{\too}{\longrightarrow} 

\newtheorem*{quest}{Questions}

\newtheorem{thm}{Theorem}[section]
\newtheorem{prop}[thm]{Proposition}
\newtheorem{lem}[thm]{Lemma}

\theoremstyle{definition}
\newtheorem{defn}[thm]{Definition}
\newtheorem*{rem}{Remark}
\newtheorem*{nota}{Notation}

\title{Growth and displacement of free product automorphisms}

\author{Matthew Collins}

\begin{document}
	
	\begin{abstract}

		It is well known for an \emph{irreducible } free group automorphism that its growth rate is equal to the minimal Lipschitz displacement of its action on Culler-Vogtmann space. This follows as a consequence of the existence of train track representatives for the automorphism.
		
		We extend this result to the general - possibly reducible - case as well as to the free product situation where growth is replaced by `relative growth'.   
	\end{abstract}

	\maketitle
	
	\tableofcontents
	
	\section{Introduction}

	This paper was prompted by a question which arose during one of the author's previous papers \cite{Collins2023}. That work was a generalisation of the results of \cite{Dicks1993} from free groups to free products, and focused on the deformation space $\mathcal{O}(G,\mathcal{G})$ - a space of $G$-trees which can be thought of as a generalisation of Culler-Vogtmann space. In \cite{Dicks1993}, the authors focused on irreducible, \emph{growth rate 1} automorphisms of free groups, where the growth rate was defined with respect to word length. However, when generalising their results, we ultimately ended up using irreducible, \emph{displacement 1} automorphisms of free products - the displacement being a value used to describe the action of an automorphism on $\mathcal{O}(G,\mathcal{G})$ (described below).
	
	Our new results held regardless of this distinction, and the link between the growth rate and the displacement of an irreducible free group automorphism is reasonably intuitive given a solid understanding of Culler-Vogtmann space and train track maps. The goal of this paper is to turn that intuition into solid proof, and to extend that proof to a larger class of automorphisms by answering the following:

	\begin{quest}
		
	\medspace
		
		\begin{enumerate}[(1)]
			\item In \cite{Dicks1993}, the growth rate was defined for free groups. What is the ``correct" generalisation of this definition to free products?
			\item When we have found this generalisation, does the growth rate of an \emph{irreducible} free product automorphism equal its displacement in outer space?
			\item More generally, does the growth rate of \emph{any} free product automorphism equal its displacement in outer space?
		\end{enumerate}
	\end{quest}

	\addtocontents{toc}{\protect\setcounter{tocdepth}{1}}
	\subsection*{(1)}

	First we address the definition of the growth rate. In the case of free groups, Dicks \& Ventura \cite{Dicks1993} attribute the following definition of the growth rate in free groups to Thurston:
	
	\begin{restatable*}{defn}{oldgrowthrate}\label{defn:og_growth_rate}
		Let $\a$ be an automorphism of a finitely generated group $G$. Let $E$ be a finite generating set of $G$, and let $l_E$ denote the conjugacy length in the alphabet $E$. Then for $g\in G$ we define the \emph{growth rate of $\a$ with respect to (the conjugacy class of) $g$} as
		\begin{align*}
			\text{GR}(\a, l_E, g) = \limsup_{k\to \infty} \sqrt[k]{l_E(g \a^n)}
		\end{align*}
		The \emph{growth rate of $\a$} is then
		\begin{align*}
			\text{GR}(\a, l_E) = \sup\{\text{GR}(\a,g) \mid g\in G\}
		\end{align*}	
	\end{restatable*}
	
	Using this definition, it can be fairly swiftly proved that the growth rate of an irreducible  automorphism of a free group is equal to its displacement in Culler-Vogtmann space. We will say more about this proof when we talk about Question (2), but when we say that we are searching for the ``correct" definition of growth rate in free products, it is the generalisation of this proof we have in mind.
	
	We elect to use \emph{relative generating sets}, which was inspired by Osin's paper on relatively hyperbolic groups \cite{osin2006}.

	\begin{restatable*}{defn}{relgenset}
		We say that $E\subseteq G$ is a \emph{relative generating set of $G$ with respect to $\mathcal{G}$} if $G$ is generated by the set
		\begin{align*}
			\bigg( \bigcup_{i=1}^k G_i \bigg) \cup E
		\end{align*}
		Where $G = G_1*\ldots*G_k*F_r$ is a free product decomposition corresponding to $\mathcal{G}$.
	\end{restatable*}
	
	From this we can essentially follow the definition of the growth rate in free groups by defining \emph{relative conjugacy length}, \emph{relative Lipschitz equivalence}, and finally the \emph{relative growth rate}.
	
	\begin{restatable*}{defn}{newgrowthrate}\label{defn:relative_growth_rate}
		Let $\a \in \text{Out}(G,\mathcal{G})$, and let $l_E$ be a relative conjugacy length function. Then for $g\in G$ we define the \emph{relative growth rate of $\a$ with respect to (the conjugacy class of) $g$} as
		\begin{align*}
			\text{GR}_{\mathcal{G}}(\a, l_E, g) = \limsup_{k\to \infty} \sqrt[k]{l(g \a^k)}
		\end{align*}
		The \emph{relative growth rate of $\a$} is then
		\begin{align*}
			\text{GR}_{\mathcal{G}}(\a, l_E) = \sup\{\text{GR}(\a,g) \mid g\in G\}
		\end{align*}
	\end{restatable*}
	
	\begin{rem}
		If we consider the case where $G$ is a free group and $\mathcal{G}$ is the trivial free factor system (that is to say, we take the free product decomposition of $G$ consisting of a single free factor - $G$ itself), then all of these ``relative" definitions restrict to their original counterparts, as they should.
	\end{rem}

	\subsection*{(2)} Questions (2) and (3) are both true. In fact, (3) immediately implies (2), but the case of irreducible automorphisms is different enough to deserve separate consideration.
	
	A free product decomposition of a group $G$ determines a set of subgroups $\mathcal{G}$ called a free-factor system. We write $\mathcal{O} = \mathcal{O}(G,\mathcal{G})$ to denote the space of equivalence classes of minimal, cocompact, edge-free $G$-trees whose set of elliptic subgroups is $\mathcal{G}$. Such a space is referred to as \emph{outer space}. $\text{Out}(G,\mathcal{G})$ - the group of outer automorphisms which preserve $\mathcal{G}$ - acts on outer space by ``twisting" the action, and studying the effect of this twist gives us information about the original automorphisms.

	For any two $G$-trees $T,S\in\mathcal{O}$, we write $\Lambda_R(T,S)$ to denote the asymmetric Lipschitz distance, or stretching factor, between them. For an automorphism $\a\in\text{Out}(G,\mathcal{G})$, one can define the \emph{displacement} of $\a$ as $\lambda_\a = \inf\{\Lambda(T,\a T)\mid T\in \mathcal{O}\}$. The \emph{minimally displaced set} of $\a$, $\text{Min}(\a)$, is the set of $G$-trees $T$ in outer space which realise this infimum.
	
	It is well known for an irreducible \emph{free group} automorphism that its growth rate is equal to the minimal Lipschitz displacement of its action on Culler-Vogtmann space. This follows as a consequence of the existence of train track representatives for the automorphism. It can be shown that the required properties of train track representatives in free groups also hold in free products, so the same method can be used.
	
	More specifically, \cite{francavigliamartino2015} proves that:
	\begin{itemize}
		\item The minimally displaced set $\text{Min}(\a)$ is equal to the set of trees in $\mathcal{O}$ which support optimal train track maps $f:T\to \alpha T$
		\item When $\alpha$ is irreducible, $\text{Min}(\a)$ is non-empty.
	\end{itemize}
	From here the proof of (2) follows. We give a complete proof of this irreducible case in Appendix 2.

	\subsection*{(3)} If we drop the irreducibility condition, a problem arises which prevent us from copying the train track proof outright: we cannot guarantee that $\text{Min}(\alpha)$ will be non-empty, and hence we cannot guarantee the existence of an optimal train track map. Happily, however, we can guarantee the existence of a weaker set of maps known as \emph{relative train tracks}.
	
	

	A topological representative $f:T\mapsto \a T$ has an associated \emph{transition matrix} $M= (m_{ij})$, where $m_{ij}$ is the number of times the $f$-image of the $j$-th edge-orbit crosses the $i$-th edge-orbit. By relabelling edges appropriately, it is always possible to write $M$ in block upper triangular form:
	
	$M =
	\left( {\begin{array}{cccc}
			M_1 & ? & ? & ?\\
			0 & M_2 & ? & ?\\
			\vdots & \vdots & \ddots & \vdots\\
			0 & 0 & \cdots & M_n\\
	\end{array} } \right)$
	
	where the matrices $M_1,\ldots M_n$ are either zero matrices or irreducible matrices.
	
	Writing $M$ in this form determines a partition of the edges of $T$: The \emph{$r$th stratum} $H_r$ of $T$ is the subgraph of $T$ given by closure of the union of the edge orbits corresponding to the rows/columns in $M_r$. We say an edge path $\gamma$ in $T_r$ is \emph{$r$-legal} if no component of $\gamma\cap H_r$ contains an illegal turn.
	
	\begin{restatable*}{defn}{RTTmapdefn}[Relative train track] \label{defn:RTTmap}
		Let $T\in \mathcal{O}(G,\mathcal{G})$, let $\alpha \in \text{Out}(G,\mathcal{G})$, and let $f:T\to \alpha T$ be a simplicial topological representative for $\a$. Use this map to divide $T$ into strata as described above. We say that $f$ is a \emph{relative train track map} if the following hold:
		\begin{itemize}
			\item[(1)] \emph{$f$ preserves $r$-germs}: For every edge $e\in H_r$, the path $f(e)$ begins and ends with edges in $H_r$.
			\item[(2)] \emph{$f$ is injective on $r$-connecting paths}: For each nontrivial path $\gamma\in T_{r-1}$ joining points in $H_r \cap T_{r-1}$, the homotopy class $[f(\gamma)]$ is nontrivial.
			\item[(3)] \emph{$f$ is $r$-legal}: If a path $\gamma$ is $r$-legal, then $f(\gamma)$ is $r$-legal.
		\end{itemize}
	\end{restatable*}
	
	\begin{restatable*}{thm}{RTTmapsexist}\cite[Thm 2.12]{Collins1994}\hspace{25pt}
		
		For any automorphism $\a \in\text{Out}(G,\mathcal{G})$, there exists a relative train track map $f:T\to \a T$ on some $T\in \mathcal{O}$.
	\end{restatable*}

	We observe that $M_r$ is the transition matrix of $H_r$, and each of these submatrices will have its own PF-eigenvalue $\mu_r$. It can be shown that, even though a relative train track map will not, in general, satisfy the property of train track maps used in Question (2), it \emph{will} satisfy a similar property on each stratum of $T$ using these $\mu_r$. This gives us the tools we require to prove that question (3) is true:
	
	\begin{restatable*}{thm}{mainThm}
		Let $\a\in\text{Out}(G,\mathcal{G})$. Then the following are equal:
		\begin{itemize}
			\item The relative growth rate of $\a$, $\text{GR}_\mathcal{G}(\a)$.
			\item The largest PF-eigenvalue $\mu_R$ of any relative train track map $f:T\to \a T$, $T\in\mathcal{O}(G,\mathcal{G})$.
			\item The displacement $\lambda_\a$ of $\a$ in $\mathcal{O}$.
		\end{itemize}
	\end{restatable*}
	
	We prove this by proving three inequalites:
	\begin{itemize}		
		\item[\textbf{A:}] $\mu_R \leq \text{GR}_\mathcal{G}(\a)$
		
		\item[\textbf{B:}] $\text{GR}_\mathcal{G}(\a) \leq \lambda_\a$
		
		\item[\textbf{C:}] $\lambda_\a \leq \mu_R$
	\end{itemize}
	
	\textbf{A} follows from existing properties of relative train tracks. \textbf{B} can be proved explicitly using the definition of the right stretching factor. \textbf{C} requires more thought.
	
	Recall the definition of the displacement: $\lambda_\a := \inf_{S\in\mathcal{O}}\Lambda_R(S,\a S)$. Ideally we would prove inequality $\textbf{C}$ by emulating the proof of question (2) and finding a $G$-tree in $\mathcal{O}$ whose right stretching factor is exactly $\mu_R$. However, unless $\a$ is irreducible, this is not always possible. Thus we instead find a \emph{sequence} of $G$-trees whose right stretching factors \emph{tend} towards $\mu_R$.
	
	The lengths of edges in $T$ are determined by the PF-eigenvectors, but these are only determined up to scalar multiplication, so we are free to rescale the edges in each stratum by a constant of our choosing. We choose to rescale each $H_r$ by $N^r$. As $N$ tends to infinity, we observe that the growth in the stratum with the largest PF-eigenvalue becomes greater than that of all other strata. From this the result follows.

	\addtocontents{toc}{\protect\setcounter{tocdepth}{2}}
	\section{Bass-Serre Theory}
	
	\subsection*{- $G$-trees}\hspace{1pt} \\

	\begin{defn}
		A \emph{simplicial tree} is a non-empty, 1-dimensional simplicial complex in which every two points are joined by a unique arc. We call the 1-simplices \emph{edges}, and the 0-simplices \emph{vertices}.
		
		A \emph{metric simplicial tree} is a simplicial tree $T$ together with a metric $d_T$ such that the set of vertices is discrete in the topology induced by $d_T$.
	\end{defn}
	
	\begin{rem}
		The discreteness condition above is equivalent to saying that the lengths of the edges are locally bounded below - that is to say, $\forall v\in V, \exists C>0$ such that for all vertices $w$ adjacent to $v$, $d_T(v,w)\geq C$.
	\end{rem}
	
	Let $G$ be a group.
	
	\begin{defn}
		A \emph{$G$-tree} is a triple $(T, d_T, \cdot)$, where $T$ is a metric simplicial tree, $d_T$ is the metric on $T$, and $\cdot$ is an isometric group action $T\times G \to T, (x, g) \mapsto x \cdot g$. For the sake of brevity of notation, if the specific metric and action are not required, we shall simply denote the triple $(T, d_T, \cdot)$ by $T$.
		
		We say that two $G$-trees are \emph{equivalent} if there exists an equivariant isometry between them.
		
		We say that $T$ is \emph{minimal} if it does not contain a $G$-invariant subtree.
		
		We say that $T$ is \emph{edge-free} if every edge has trivial stabiliser.
	\end{defn}
	
	\begin{rem}
		We say that an action on a $G$-tree $T$ is \emph{without inversions} if no edge is sent to its inverse by an element of $G$. If $T$ does contain an inverted edge, then placing a new vertex at the midpoint of will essentially remove this inversion; Furthemore, since the two trees before and after this operation are equivalent in the above sense, adding this vertex does not affect any relevant properties of $T$. Thus, to simplify our calculations, we shall assume that all our $G$-trees are without inversions.
	\end{rem}
	
	\begin{defn}\label{defn:ellip-hyp}
		Let $(T, d_T, \cdot)$ be a $G$-tree.
		
		We say that an element $g\in G$ is \emph{elliptic} with respect to $T$ if $x\cdot g = x$ for some point $x\in G$. If $g$ is not elliptic, we say it is \emph{hyperbolic}.
		
		We say that a subgroup $H\leq G$ is \emph{elliptic} with respect to $T$ if $x\cdot H = x$ for some point $x\in G$. If $H$ is not elliptic, we say it is \emph{hyperbolic}.
		
		An elliptic subgroup will consist entirely of elliptic elements, but the converse is not necessarily true.
	\end{defn}
	
	\begin{defn}
		We write $\mathcal{O} = \mathcal{O}(G,\mathcal{G})$ to denote the space of equivalence classes of minimal, cocompact, edge-free $G$-trees whose set of elliptic subgroups is $\mathcal{G}$. Such a space is referred to as \emph{outer space}.
	\end{defn}
	
	\subsection*{- Graphs of groups}
	
	\begin{defn}
		A \emph{(Serre) graph} $X$ consists of the following:
		\begin{itemize}
			\item A vertex set $V=V(X)$
			\item An edge set $E=E(X)$
			\item An initial vertex map $\i:E\to V$
			\item An edge reversal map $E\to E, e\mapsto \overline{e}$ such that $e\neq \overline{e}$ and $\overline{\overline{e}} = e$
		\end{itemize}
		We call $\i(\overline{e})$ the \emph{terminal vertex} of $e$, and denote it $\tau(e)$.
	\end{defn}
	
	\begin{defn}\cite[p.198]{cohen} A \emph{graph of groups} $\G$ consists of the following:
		\begin{itemize}
			\item[(i)] A connected graph $X$.
			\item[(ii)] A group $G_v$ for each vertex $v$ of $X$, and a group $G_e$ for each edge $e$ of $X$ such that $G_{\overline{e}} = G_e$.
			\item[(iii)] For each edge $e$ of $X$, a monomorphism $\rho_e : G_e \to G_{\tau e}$.
		\end{itemize}
		If $Y$ has a metric, we say that $\G$ is a \emph{metric graph of groups}.
	\end{defn}
	
	\begin{nota}
		We may write $V(\G)$ and $E(\G)$ to denote the sets $V(X)$ and $E(X)$ respectively.
	\end{nota}
	
	\begin{rem}
		All of the graphs of groups we shall be using in this paper will have trivial edge groups (and hence trivial monomorphisms $\rho_e$). Thus, for the sake of simplicity, we shall restrict our exposition on Bass-Serre Theory to graphs of groups with this property.
	\end{rem}
	
	Let $\G$ be a graph of groups with trivial edge groups. Then the \emph{path group} $\pi(\G)$ is defined by
	\begin{align*}
		\pi(\G) = \frac{\big(\Asterisk_{v\in V(\G)}G_v\big)*F(E(\G))}{N}
	\end{align*}
	
	where $N$ is the normal closure of the set $\{e\overline{e}\mid e\in E(\G)\}$.
	
	A \emph{path} in $\G$ is a sequence $g_0 e_1 g_1 \ldots g_{n-1}e_n g_n$ where $g_{i-1}\in G_{\iota e_i}$ and $g_i\in G_{\tau e_i}$ for all $i$ (so $e_1\ldots e_n$ is an edge path in the graph). If any $g_i = 1$, then we omit it from the notation.
	
	We say a path is \emph{reduced} if it does not contain any subpaths of the form $e \overline{e}$. We can think of $\pi(\G)$ as the group of reduced paths.
	
	\begin{defn}
		Choose a base point $x\in V(\G)$. The \emph{fundamental group} $\pi_1(\G,x)$ of $\G$ at $x$ is the subgroup of $\pi(\G)$ consisting of the reduced paths which start and end at $x$.
	\end{defn}
	
	The isomorphism class of $\pi_1(\G,x)$ does not depend on our choice of $x$. Thus we shall usually omit it from the notation, and simply write $\pi_1(\G)$.
	
	If $X$ denotes the underlying graph of $\G$, then it can be shown that
	\begin{align*}
		\pi_1(\G) \cong \big(\Asterisk_{v\in V(\G)}G_v\big)*F_r
	\end{align*}
	where $F_r \cong \pi_1(X)$ is a free group of rank $r$.
	
	\begin{defn}
		Let $G$ be a group, and let $\G$ be a metric graph of groups. A \emph{marking} on $\G$ is an isomorphism $\phi: G \to \pi_1(\G)$. The pair $(\G,\phi)$ is called a \emph{marked graph of groups}.
	\end{defn}
	
	\subsection*{- The Fundamental Theorem}
	
	Bass-Serre theory describes a process by which one may construct a marked graph of groups from a $G$-tree and, conversely, a $G$-tree from a marked graph of groups. The Fundamental Theorem of Bass-Serre Theory states that these two constructions are mutually inverse, up to isomorphism of the structures involved.
	
	The details of Bass-Serre Theory have been thoroughly explored in the literature (\cite{Bass1993} \cite{cohen}), so we shall simply give a brief description of the two constructions.
	
	\begin{defn}[Bass-Serre Tree]\cite[p.7]{Andrew2021}
		
		Let $(\G,\phi)$ be a marked graph of groups with trivial edge groups and with marking $\phi: G \to \pi_1(\G,x)$. We then define a graph $T$ called the \emph{universal cover}, or \emph{Bass-Serre tree}, of $\G$ as follows:
		\begin{itemize}
			\item The vertex set $V(T)$ is the set of `cosets' $G_v\gamma$, where $\gamma \in \pi(\G)$ is a path from the vertex $v$ to our base point $x$.
			\item Two vertices $G_{v_1}\gamma_1, G_{v_2}\gamma_2 \in V(T)$ are joined by an edge(-pair) if $\gamma_1 = e g_{v_2} \gamma_2$ or $\gamma_2 = e g_{v_2} \gamma_1$ for some edge $e$ and some $g_{v_i}\in G_{v_i}$.
		\end{itemize}
		It can be shown that this graph $T$ is always a tree.
		
		If $g\in G$, then $\phi(g)\in \pi_1(\G,x)$ is a loop in $\G$ - that is to say, a path from $x$ to $x$. Thus we can define a right action $\cdot$ of $G$ on $T$ as follows:
		\begin{align*}
			\forall G_v\gamma \in V(T),\hspace{35pt} G_v\gamma\cdot\phi(g) := G_v (\gamma \phi(g))
		\end{align*}
		This action respects adjacency, sending edges to edges. Together with this action, the Bass-Serre tree is an edge-free $G$-tree.
	\end{defn}
	
	We say two marked graphs of groups are \emph{equivalent} if their universal covers are equivalent as $G$-trees.
	
	\begin{rem}
		Observe that the Bass-Serre tree as we have defined it here is a Serre graph, whereas the definition of $G$-trees we have given views them as simplicial structures. This distinction ultimately matters little - we can reconcile the two viewpoints by thinking of each pair $(e,\overline{e})$ as denoting two orientations of a 1-simplex, rather than being individual edges.
	\end{rem}
	
	\begin{defn}[Quotient Graph of Groups]
		Let $T$ be an edge-free $G$-tree. Then we define of \emph{quotient graph of groups} of $T$ as follows:
		\begin{itemize}
			\item The underlying Serre graph is the quotient graph $T/G$
			\item All edge groups are trivial.
			\item Consider a connected fundamental domain in $T$. This will contain exactly one vertex from each orbit. We assign the stabilizers of these vertices to be the vertex groups of the corresponding vertices in $T/G$.
		\end{itemize}
		It is given by the Fundamental Theorem of Bass-Serre Theory that the fundamental group of this graph of groups is isomorphic to $G$. The action of $G$ on $T$ determines this isomorphism, giving us a marking. 
	\end{defn}
	
	\subsection*{- Free factor systems}
	
	There exists a relation between $G$-trees and free product decompositions of $G$, which can be phrased as follows:
	
	Consider outer space $\mathcal{O} = \mathcal{O}(G,\mathcal{G})$, and recall that trees in this space are cocompact. A $G$-tree is cocompact if and only if it has a finite number of edge orbits and vertex orbits. It follows that a quotient graph of groups $\G$ has a finite number of vertex groups, and hence we can write $G = \pi_1(\G) \cong G_1 * \ldots * G_k * F_r$, where $G_1, \ldots, G_k$ are the vertex groups. (Since we have a choice of vertex groups when constructing $\G$, this free product decomposition is only unique up to conjugates of the free factors).
	
	Conversely, suppose we are given a free product $G = G_1 * \ldots * G_k * F_r$. Then one can easily construct graphs of groups with vertex groups $G_1, \ldots, G_k$ and fundamental group $G$ (for example, see Figure \ref{fig:simpleGOG}). Taking the Bass-Serre trees will give us an outer space $\mathcal{O} = \mathcal{O}(G,\mathcal{G})$ of $G$-trees, where the elliptic subgroups are $\mathcal{G} = \langle H\leq G_i^g \mid g\in G, i=1,\ldots, k \rangle$ - the subgroups of the conjugates of the free factors.
	
	To summarize, a space $\mathcal{O}(G,\mathcal{G})$ determines a family of free product decompositions of $G$, and a free product decomposition of $G$ determines a space $\mathcal{O}(G,\mathcal{G})$. It follows from the Fundamental Theorem of Bass-Serre Theory that these two processes are mutually inverse (up to conjugacy of the free factors). Thus we shall refer to the set $\mathcal{G}$ as a \emph{free factor system}.
	
	\begin{defn}
		Let $\mathcal{G}$ be a free factor system for a group $G$, and let $g\in G$.
		
		We say that $g$ is \emph{$\mathcal{G}$-elliptic} if it lies in a subgroup in $\mathcal{G}$. Otherwise, we say it is \emph{$\mathcal{G}$-hyperbolic}. We shall write $\text{Hyp}(\mathcal{G})$ to denote the set of all $\mathcal{G}$-hyperbolic elements.
	\end{defn}
	
	\begin{rem}
		The notions of elliptic and hyperbolic given here align with those in Definition \ref{defn:ellip-hyp}.
	\end{rem}
	
	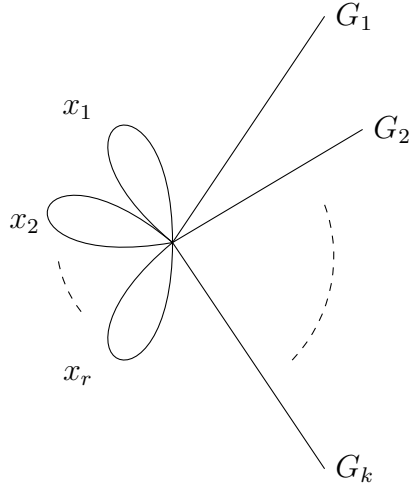
\begin{figure}
		\begin{tikzpicture}[scale = 5]
			\draw (0,0) -- (0.4,0.6) node[anchor=west] {$G_1$};
			\draw (0,0) -- (0.5,0.3) node[anchor=west] {$G_2$};
			\draw (0,0) -- (0.4,-0.6) node[anchor=west] {$G_k$};
			
			\draw (0,0)  to[in=90,out=140,loop] (0,0);
			\draw (0,0)  to[in=140,out=190,loop] (0,0);
			\draw (0,0)  to[in=-140,out=-90,loop] (0,0);
			
			\draw (-0.18,0.3) node[anchor = south east] {$x_1$};
			\draw (-0.32,0.05) node[anchor = east] {$x_2$};
			\draw (-0.18,-0.3) node[anchor = north east] {$x_r$};
			
			\draw[dashed] (0.4,0.1) arc (20:-45:0.4);
			\draw[dashed] (-0.3,-0.05) arc (190:220:0.3);
		\end{tikzpicture}
		\label{fig:simpleGOG}\caption[Figure 1]{A graph of groups with fundamental group $G_1 * \ldots G_k * \langle x_1, \ldots, x_r \rangle$}
	\end{figure}

	\section{Length functions and growth}
	
	\begin{defn}
		By a \emph{length function} on a set $X$, we mean a map $l:X\to \R$ taking non-negative values.
	\end{defn}
	
	\subsection*{- Displacement in Outer Space}
	
	\begin{defn}
		Let $g\in G$, and let $T\in \mathcal{O}(G,\mathcal{G})$. Then we write $l_T(g)$ to denote the \emph{translation length} of $g$ in $T$, given by
		\begin{align*}
			l_T(g) = \inf_{x\in T}\{ d_T(x,x\cdot g) \}
		\end{align*}
	\end{defn}
	
	\begin{rem}
		$l_T(g^h) = l_T(g)$ for all $g,h\in G$. Thus we shall think of $l_T$ as a length function on the conjugacy classes of of single elements of $G$.
		
		It can be shown that this infimum is achieved by some $x\in T$. If $g$ is elliptic, then $l_T(g) = 0$ by definition.
		
		If $g$ is hyperbolic, then $l_T(g) > 0$. The set of points $x\in T$ which realise $l_T(g)$ will form a line in $T$ called the \emph{hyperbolic axis of $g$}. Points on the axis will be translated along the axis a distance of $l_T(g)$ by $g$.
	\end{rem}

	\begin{defn}\cite[p.8]{francavigliamartino2015}\label{def:stretching_factors}
		Let $T,S\in \mathcal{O}(G,\mathcal{G})$. We define the \emph{right stretching factor} from $T$ to $S$ as
		\begin{align*}
			\Lambda_R(T,S) := \sup_{g\in \text{Hyp}(\mathcal{G})}\dfrac{l_S(g)}{l_T(g)}
		\end{align*}
	\end{defn}
	
	\begin{defn}
		Let $\a\in \text{Out}(G,\mathcal{G})$. Then we define the \emph{displacement} of $\a$ to be
		\begin{align*}
			\l_\a :=\inf_{T\in\mathcal{O}}\Lambda_R(T,\a T)
		\end{align*}
	\end{defn}

	\subsection*{- Growth rate}\hspace{1pt}
	
	Dicks \& Ventura \cite{Dicks1993} attribute the following definition to Thurston.
	
	\oldgrowthrate
	
	\begin{rem}
		It can be shown that any two finite generating sets of a finitely generated group $G$ give Lipschitz equivalent conjugacy length functions, and from this it can be further shown that Definition \ref{defn:og_growth_rate} does not depend on the choice of $E$.
	\end{rem}
	
	We shall now generalise this definition from finitely generated groups to finite free products - that is to say, free products of the form $G = G_1 * \ldots * G_k * F_r$, where $k$ and $r$ are finite, and $F_r$ denotes a free group of rank $r$.

	\subsection*{- Relative length functions} \hspace{1pt} \\

	\relgenset
	
	Let $\hat{G} :=  \bigsqcup_{i= 1}^k (G_i \backslash  \{1\})$. Then we observe that $G$ is generated by the set $E\cup \hat{G}$.

	\begin{defn}
		Let $E$ be a relative generating set of $G$ with respect to $\mathcal{G}$. To each element $g\in G$, we assign its \emph{relative length with respect to $\mathcal{G}$}, $|g|_{E\cup\hat{G}}$, to be the length of a shortest word in the alphabet $E\cup \hat{G}$ which represents $g$ in $G$.
		
		To each element $g\in G$, we assign the \emph{relative conjugacy length with respect to $\mathcal{G}$}, $l_E(g)$, given by 
		\begin{align*}
			l_E(g) = \min\{|h|_{E\cup\hat{G}} \mid h\in G \text{ is conjugate to } g\}
		\end{align*}
	\end{defn}
	
	It is clear that this length depends upon the choice of $E$. However, when $E$ is finite, the effect this choice has on the length is bounded in the following sense:
	
	\begin{defn}
		We say that two length functions $l_1, l_2$ on the same set $X$ are \emph{Lipschitz equivalent} if there exist constants $C,D>0$ such that, for all $x \in X$, $C l_2(x)\leq l_1(x) \leq D l_2(x)$. We write $l_1\sim l_2$.
		
		Let $G = G_1 * \ldots * G_k * F_r$. We say that two length functions $l_1, l_2$ on the conjugacy classes of $G$ are \emph{relatively Lipschitz equivalent} if they are Lipschitz equivalent when restricted to $\text{Hyp}(\mathcal{G})$. We write $l_1\sim_{\mathcal{G}} l_2$.
	\end{defn}
	
	\begin{prop}\label{prop:lip_equiv_word_lengths}(Adapted from \cite[p.13]{osin2006})
		Let $G = G_1*\ldots*G_k*F_r$ be a free product. Suppose that $E_1$ and $E_2$ are two finite relative generating sets of $G$ with respect to $\{G_1,\ldots,G_k\}$. Then the corresponding length functions $l_{E_1}$ and $l_{E_2}$ are Lipschitz equivalent (and hence relatively Lipschitz equivalent).
	\end{prop}
	
	\begin{thm}
		Let $\a\in \text{Out}(G,\mathcal{G})$, let $E$ be a finite relative generating set for $G$, and let $T\in \mathcal{O}$. Then $l_E \sim_\mathcal{G} l_T$.
	\end{thm}
	\begin{proof}
		There exists a free product decomposition $G = G_1 * \ldots * G_k * F_r$ corresponding to $E$. Let  $E' = \{x_1, \ldots, x_r\}$ be a basis for $F_r$. Then $E'$ is also a finite $\mathcal{G}$-relative generating set for $G$. Let $\G'$ denote the graph of groups depicted in Figure \ref{fig:simpleGOG}. The universal cover $T'$ of $\G'$ lies in $\mathcal{O}$.
		
		To define a metric on $\G'$ (and hence $T'$), we assign the loops on the left length 1, and the remaining edges on the right length $\frac{1}{2}$. By assigning these edge lengths, we ensure that $l_{E'}(g) = l_{T'}(g)$ for all hyperbolic $g\in G$, hence $l_{E'} \sim_{\mathcal{G}} l_{T'}$
		
		There exist Lipschitz continuous maps between any two $G$-trees in $\mathcal{O}$ \cite[Lem 4.2]{francavigliamartino2015}. This is sufficient to prove that $l_{T'} \sim_{\mathcal{G}} l_{T}$.
		
		Thus
		\begin{align*}
			l_E \stackrel{\text{Prop \ref{prop:lip_equiv_word_lengths}}}{\sim_{\mathcal{G}}} l_{E'} \sim_{\mathcal{G}} l_{T'} \sim_{\mathcal{G}} l_{T}
		\end{align*}
	\end{proof}
	
	\subsection*{- Automorphisms}
	
	\begin{nota}
		The outer automorphism group of a group $G$ is defined as $\text{Out}(G) := \text{Aut}(G)/ \text{Inn}(G)$. Elements of $\text{Out}(G)$ are equivalence classes of automorphisms, where two automorphisms are equivalent if they differ by an inner automorphism. However, when we write $\a \in \text{Out}(G)$ we mean that $\a$ is an automorphism in $\text{Aut}(G)$ representing an equivalence class in $\text{Out}(G)$.
		
		In this paper, the automophisms of $G$ will act on $G$ on the right.
	\end{nota}
	
	\begin{defn}
		Let $\mathcal{G}$ be a free factor system for a group $G$, and let $\a \in \text{Out}(G)$. We say that $\mathcal{G}$ is $\a$-invariant if $\mathcal{G} = \mathcal{G} \a$. We write $\text{Out}(G,\mathcal{G})$ to denote the subgroup of $\text{Out}(G)$ consisting of the elements $\a$ such that $\mathcal{G}$ is $\a$-invariant.
	\end{defn}
	
	\subsection*{- The relative growth rate}
	
	\begin{defn}
		Let $\mathcal{G}$ be a free factor system for a group $G$. We say a \emph{$\mathcal{G}$-bounded} length function is a a length function $l$ on the conjugacy classes of single elements of $G$ which satisfies the following:
		\begin{itemize}
			\item $\exists C >0$ such that $\forall g\notin \text{Hyp}(\mathcal{G})$, $l(g) \leq C$
			\item $\exists \e >0$ such that $\forall g\in \text{Hyp}(\mathcal{G})$, $\e \leq l(g)$
			\item $\forall g\in \text{Hyp}(\mathcal{G})$, for all non-zero integers $n$, $l(g^n) = |n|l(g)$
		\end{itemize}
	\end{defn}

	\newgrowthrate
	
	\begin{lem}\label{lem:only_hyperbolic_matters}
		Let $G, \mathcal{G}, \a$ and $l$ be as in Definition \ref{defn:relative_growth_rate}. Then the relative growth rate of $\a$ is determined by the hyperbolic elements of $G$ - that is to say,
		\begin{align*}
			\text{GR}_{\mathcal{G}}(\a, l) = \sup\{\text{GR}(\a,g) \mid g\in \text{Hyp}(\mathcal{G})\}
		\end{align*}
	\end{lem}
	\begin{proof}
		To prove this, it suffices to show there exists a hyperbolic element of $G$ whose relative growth rate is greater than or equal to the relative growth rate of every elliptic element.
		
		Let $g\in G$ be an elliptic element. $\mathcal{G}$ is $\a$-invariant, hence $g \a^k$ is also elliptic. Furthermore, since $l$ is $\mathcal{G}$-bounded, $\exists C > 0$ such that $l(g)\leq C$.
		\begin{align*}
			\Rightarrow \text{GR}_{\mathcal{G}}(\a, l, g) &= \limsup_{k\to \infty} \sqrt[k]{l(g \a^k)}\\
			&\leq \limsup_{k\to \infty} \sqrt[k]{C} &&\text{(by properties of $\limsup$)}\\
			&= 1
		\end{align*}
		
		Thus we wish to find a hyperbolic element whose relative growth rate is at least 1. Since $l$ is $\mathcal{G}$ bounded, $\exists \e >0$ such that $\forall h\in \text{Hyp}(\mathcal{G})$, $\e \leq l(g)$. Let $h\in \text{Hyp}(\mathcal{G})$, and take $N\geq \frac{1}{\e}$. Then
		\begin{align*}
			l(h^N\a^k) = l((h\a^k)^N) = |N|\cdot l(g\a^k) \geq \frac{1}{\e} \cdot \e = 1
		\end{align*}
		
		$h^N$ is also hyperbolic, hence this completes the proof.
	\end{proof}
	
	Note that we have only defined the growth rate for $\mathcal{G}$-bounded length functions. We can show that both of the length functions we have seen in this paper are $\mathcal{G}$-bounded:
	
	\begin{lem}
		Let $G= G_1 * \ldots * G_k * F_r$ be a free product with corresponding free factor system $\mathcal{G}$. Let $E$ be a relative generating set with respect to $\mathcal{G}$, and let $l_E$ be the corresponding relative conjugacy length. Then $l_E$ is a $\mathcal{G}$-bounded length function.
	\end{lem}
	\begin{proof}
		Let $g\in G$. If $g$ is elliptic, then $l_E(g) = 1$, so take $C = 1$
		
		The only element of $G$ with relative conjugacy length 0 is the identity element, which is elliptic. Thus we may take $\epsilon = 1$.
		
		Let $g$ be hyperbolic, and let $x_1\ldots x_k$ be a shortest word in $E$ representing a conjugate of $g$. If $x_1= x_k^{-1}$, then conjugating by $x_1^{-1}$ would shorten the word, resulting in a contradiction. Thus we may may assume that this is not the case, and hence
		\begin{align*}
			l_E(g^n) = l_E((x_1\ldots x_k)^n) = |n|k = |n|l_E(g)
		\end{align*}
	\end{proof}
	
	\begin{lem}
		Let $T\in \mathcal{O}(G,\mathcal{G})$. Then $l_T$ is a $\mathcal{G}$-bounded length function.
	\end{lem}
	\begin{proof}
		Elliptic elements are by definition elements which fix a point in $T$. Thus the translation length of elliptic elements is bounded above.
		
		Hyperbolic elements $g$ translate points along a hyperbolic axis - in particular, since the action is isometric, the points are translated by a whole number of edges. $T$ is cocompact, hence it contains finitely many edge orbits - thus we may take $\epsilon$ to be the length of the shortest edge.
		
		Additionally, points on the axis of $g$ are translated along it by a distance of $l_T(g)$. $g^n$ has the same hyperbolic axis, and it translates points on the axis a distance of $l_T(g^n) = |n|l_T(g)$.
	\end{proof}
	
	We will now show that these two length functions give us the same relative growth rate:

	\begin{prop}\label{prop:use_other_length_functs}
		Let $l_1, l_2$ be $\mathcal{G}$-bounded length functions, and suppose that $l_1 \sim_\mathcal{G} l_2$. Then $\text{GR}_\mathcal{G} (\a, l_1) = \text{GR}_\mathcal{G} (\a, l_2) $.
	\end{prop}
	\begin{proof}
		By Lemma \ref{lem:only_hyperbolic_matters}, it is sufficient to restrict our attention to the hyperbolic elements of $G$. Let $h\in \text{Hyp}(\mathcal{G})$. $l_1\sim_\mathcal{G} l_2$, hence $\exists D>0$ such that $l_1(h)\leq D l_2(h)$. Thus
		\begin{align*}
			\text{GR}(\a, l_1, h) &= \limsup_{n\to \infty} \sqrt[n]{l_1(h \a^n )}\\
			&\leq \limsup_{n\to \infty} \sqrt[n]{D} \sqrt[n]{l_2(h \a^n )}\\
			&\leq \Big(\limsup_{n\to \infty} \sqrt[n]{D}\Big) \Big(\limsup_{n\to \infty} \sqrt[n]{l_2(h \a^n )}\Big) &&\text{(By properties of limsup)}\\
			&= \limsup_{n\to \infty} \sqrt[n]{l_2(h \a^n )} && \text{($\limsup_{n\to \infty}\sqrt[n]{D} = \lim_{n\to \infty}\sqrt[n]{D} = 1$)}\\
			&= \text{GR}(\a, l_2, h)
		\end{align*}
		
		The reverse inequality can be obtained in the analogous way, and hence $\text{GR}(\a, l_1, h) = \text{GR}(\a, l_2, h)$. This holds for all $h\in \text{Hyp}(\mathcal{G})$, hence $\text{GR}(\a, l_1) = \text{GR}(\a, l_2)$
	\end{proof}
	
	\begin{nota}
		In a similar manner to the free group case, Proposition \ref{prop:use_other_length_functs} shows that, up to relative Lipschitz equivalence, the relative growth rate does not depend on our choice of $l$. Thus, unless the particular length function is required, we shall omit $l$ from the notation, and simply write $\text{GR}_{\mathcal{G}}(\a, g)$ and $\text{GR}_{\mathcal{G}}(\a)$.
	\end{nota}

	\section{Relative Train Tracks and Perron-Frobenius}
	
	\begin{defn}
		Let $T,S\in \mathcal{O}$. A map $f:T\to S$ is called an \emph{$\mathcal{O}$-map} if it is Lipschitz continuous and $G$-equivariant. We write $\text{Lip}(f)$ to denote the Lipschitz constant of $f$.
		
		An $\mathcal{O}$-map $f:T\to S$ is \emph{straight} if it has constant speed on edges - that is to say, the restriction of $f$ to each edge is a linear map.
	\end{defn}
	
	\begin{rem}\cite[Remark 3.4]{francavigliamartino2018a}
		$\mathcal{O}$-maps exist between any pair of $G$-trees in $\mathcal{O}$. Any $\mathcal{O}$-map $f$ can be uniquely straightened - that is to say, there exists a unique $\mathcal{O}$-map $\text{Str}(f)$ which is homotopic relative to vertices to $f$. We have $\text{Lip}(\text{Str}(f))\leq \text{Lip}(f)$.
	\end{rem}
	
	\begin{defn}
		Let $\a\in \text{Out}(G,\mathcal{G})$, and let $T\in \mathcal{O}(G,\mathcal{G})$. We call a straight $\mathcal{O}$-map $f:T\to \a T$ a \emph{topological representative for $\a$}. If $f$ maps vertices to vertices, we say it is \emph{simplicial}.
	\end{defn}
	
	\subsection*{- Stratification of $G$-trees for topological representatives}
	
	Let $f:T\to \a T$ be a simplicial topological representative for $\a\in \text{Out}(G,\mathcal{G})$. Being simplicial, $f$ will map edges in $T$ to edge paths. This behaviour determines the associated \emph{transition matrix} $M = (m_{ij})$ of $f$, where $m_{ij}$ is the number of times the $f$-image of the $j$-th edge-orbit crosses the $i$-th edge-orbit.
	
	Relabelling the edges (which equates to reordering the rows and columns of the matrix) allows us to write $M$ in block upper triangular form
	
	$M =
	\left( {\begin{array}{cccc}
			M_1 & ? & ? & ?\\
			0 & M_2 & ? & ?\\
			\vdots & \vdots & \ddots & \vdots\\
			0 & 0 & \cdots & M_n\\
	\end{array} } \right)$
	
	where the matrices $M_1,\ldots M_n$ are either zero matrices or irreducible matrices.
	
	Writing $M$ in this form determines a partition of the edges of $T$: The \emph{$r$th stratum} $H_r$ of $T$ is the closure of the union of the edge orbits in $T$ corresponding to the rows/columns in $M_r$. We observe that $M_r$ is the transition matrix of $H_r$.
	
	\begin{rem}
		Dividing $T$ into strata in this way also gives us a filtration $\emptyset = T_0\subset \ldots \subset T_n = T$ of $T$, where $T_r = \bigcup_{i\leq r} H_r$.
		
		Examining $M$ tells us that this filtration is $f$-invariant. The strata themselves are not; the $f$-image of an edge in one stratum may intersect the lower strata.
	\end{rem}
	
	\begin{thm}[Perron-Frobenius]
		Let $A$ be a non-negative, irreducible, integer-valued square matrix. Then one of its eigenvalues, called the \emph{PF-eigenvalue}, is a positive real number $\mu$ whose absolute value is greater than or equal to that of all other eigenvalues. There is a positive real eigenvector corresponding to $\mu$.
	\end{thm}
	
	\begin{defn}
		Let $f$ and $M$ be as above. We shall write $\mu_r$ to denote the PF-eigenvalue of $M_r$. We say that $H_r$ is a \emph{growing stratum} if $M_r$ is not a zero matrix.
	\end{defn}
	
	Observe that, since $f$ is non-trivial, there will always exist at least one growing stratum.
	
	\subsection*{- Relative train tracks}
	
	Relative train track maps are a particular type of simplicial topological representative which were introduced by Bestvina \& Handel in the case of free groups \cite{Bestvina1992}, and were generalised to free products by Collins \& Turner \cite{Collins1994}. Collins and Turner defined their maps on graphs of complexes - graphs with 2-complexes assigned to the vertices. However, this definition can be transferred to $G$-trees by replacing each 2-complex with its fundamental group to give a graph of groups and then lifting to the Bass-Serre tree.

	\begin{defn}
		Let $T\in \mathcal{O}$, and let $v$ be a vertex in $T$. A \emph{turn} at $v$ is a pair of directed edges $(e_1,e_2)$ such that $\tau(e_1) = v = \iota(e_2)$. We say the turn is \emph{degenerate} if $e_2 = \overline{e_1}$.
		
		A simplicial topological representative $f:T\to \a T$ induces a map $Df$ on the turns of $T$: $Df(e_1,e_2)$ is the turn consisting of the first edges in the edge paths $f(e_1)$, $f(e_2)$. A turn is \emph{illegal} with respect to $f$ if its image under some iterate of $Df$ is degenerate. Otherwise, it is \emph{legal}.
		
		We say an edge path $\gamma$ in $T$ is \emph{legal} if it does not contain any illegal turns.
		
		We say an edge path $\gamma$ in $T_r$ is \emph{$r$-legal} if $\gamma\cap H_r$ does not contain any illegal turns.
	\end{defn}
	
	\RTTmapdefn

	\RTTmapsexist

	Let $f:T \to \a T$ be a relative train track map. Then, by definition, $f$ is a simplicial topological representative for $\a$, so we may stratify $T$ as described above. Let $M$ be the transition matrix for $f$, with submatrices $M_1,\ldots, M_n$. Let $\emptyset = T_0\subset \ldots \subset T_n = T$ be the corresponding filtration, and let $H_1,\ldots, H_n$ denote the strata.
	
	For each $r$ we shall write $\mu_r$ to denote the Perron-Frobenius eigenvalue for $M_r$. Let $\mathbf{v}_r$ be the corresponding positive real \emph{row} eigenvector. This eigenvector determines \emph{$r$-lengths} $L_r(e)$ which we can assign to the edges of $H_r$: we declare the $r$-length of the $i$th edge of $H_r$ to b the $i$th entry in $\mathbf{v}_r$.
	
	If $\gamma$ is an edge path, then we define its $r$-length to be
	\begin{align*}
		L_r(\gamma) = \sum_{e\in\gamma\cap H_r} L_r(e)
	\end{align*}
	
	\begin{nota}
		We will adapt this notation slightly when considering elements of the group $G$:
		
		Recall that a hyperbolic element $g\in \text{Hyp}(\mathcal{G})$ corresponds to a hyperbolic axis in $T$, and that $l_T(g)$ is the distance points on the axis are translated by $g$. However, we can also think of $l_T(g)$ as the length of a path - specifically the length of a fundamental domian of the axis. We can always choose this fundamental domain such that it consists of whole edges (i.e. it is an edge path)
		
		Let $\gamma_g$ denote such a  fundamental domain. We say that $g$ is $r$-legal if $\gamma_g$ is $r$-legal. This is independant of our choice of $\gamma_g$. We will write $L_r(g)$ to denote the $r$-length of $\gamma_g$.
	\end{nota}
	
	Now, let us give the properties we shall require.
	
	\begin{lem}\label{lem:propertiesofRTTs}
		Let $f:T\to \a T$ be a relative train track map on some $T \in \mathcal{O}(G,\mathcal{G})$, and let $H_r$ be a growing stratum in $T$. Then the following hold:
		\begin{itemize}
			\item[(i)] Every edge in $H_r$ is $r$-legal
			\item[(ii)] If an edge path $\gamma$ is $r$-legal, then $f(\gamma)$ is $r$-legal.
			\item[(iii)] There exists an $r$-legal group element $g\in \text{Hyp}(\mathcal{G})$.
			\item[(iv)] If an edge path $\gamma$ is $r$-legal, then $L_r(f(\gamma)) = \mu_r L_r(\gamma)$.
			\item[(v)] If $g$ is $r$-legal, then $L_r(g \a^k) = \mu_r^k L_r(g)$.
		\end{itemize}
	\end{lem}
	\begin{proof}\hspace{1pt}
		\begin{itemize}
			\item[(i)] Edges do not contain any turns - in particular, they do not contain any illegal turns. Thus they are $r$-legal.
			\item[(ii)] This is one of the defining properties of a relative train track map (Definition~\ref{defn:RTTmap} (iii)).
			\item[(iii)] By (i), edges in $H_r$ are $r$-legal. Therefore, by (ii), iterating $f$ will give us longer and longer $r$-legal paths. Since $T$ is cocompact, it contains finitely many edge orbits. Thus we will eventually reach an $r$-legal path $f^k(e)$ which crosses some edge orbit at least three times. It follows that $f^k(e)$ must contain $\gamma_g$ for some hyperbolic element $g$.
			\item[(iv)] Follows from the definition of a relative train track map \cite[p454]{Collins1994}.
			\item[(v)] Since $f$ is a topological representative for $\a$, we have that $f^k(g) = g \a^k$ for all $k \geq 1$. Thus, by property (iv), $L_r(g \a) = \mu_r L_r(g)$, and property (ii) allows us to iterate $f$, giving us $L_r(g \a^k) = \mu_r^k L_r(g)$.
		\end{itemize}
	\end{proof}

	\section{Main Theorem}
	
	\mainThm
	
	\begin{proof}
		The existence of $f$ is assured by Theorem \ref{thm:RTTmapsexist}. We observe that multiple strata in $T$ may have PF-eigenvalue equal to $\mu_R$. Therefore, when we write $H_R$, we are referring to the highest of these strata - that is to say, we are maximizing the size of $T_R$. $H_R$ will always be a growing stratum.
		
		We prove this theorem by proving the following three inequalities:
		\begin{itemize}		
			\item[\textbf{A:}] $\mu_R \leq \text{GR}_\mathcal{G}(\a)$
			
			\item[\textbf{B:}] $\text{GR}_\mathcal{G}(\a) \leq \lambda_\a$
			
			\item[\textbf{C:}] $\lambda_\a \leq \mu_R$
		\end{itemize}
		
		\begin{itemize}
			\item[\textbf{A:}] By Lemma \ref{lem:propertiesofRTTs} (v), there exists an $R$-legal $g \in \text{Hyp}(\mathcal{G})$ such that $L_R(g \a^k) = \mu_R^k L_R(g)$ for all $k\geq 1$. Hence
			\begin{align*}
				\Rightarrow \text{GR}_\mathcal{G}(\a,g) &= \limsup_{k\to \infty}\sqrt[k]{l_T(g \a^k)}\\
				&\geq \limsup_{k\to \infty}\sqrt[k]{L_R(g \a^k)}\\
				&= \limsup_{k\to \infty}\sqrt[k]{\mu_R^k L_R(g)} &&\text{(by Lemma \ref{lem:propertiesofRTTs} (v))}\\
				&= \mu_R
			\end{align*}

			\item[\textbf{B:}] Let $S\in \mathcal{O}$, and for brevity of notation let $\Lambda$ denote the right stretching factor $\Lambda_R(S,\a S)$. Then
			\begin{align*}
				\Lambda_R(S,\a^k S) &\leq \Lambda_R(S,\a S)\Lambda_R(\a S,\a^2 S)\ldots \Lambda_R(\a^{k-1} S,\a^k S)\\
				&= \Lambda_R(S,\a S)^k &&\text{(by triangle inequality)}\\
				&= \Lambda^k\\
				\Rightarrow \frac{l_S(g \a^k)}{l_S(g)} &\leq \Lambda^k \hspace{35pt} \forall g\in G &&\text{(by definition of $\Lambda_R$)}\\
				\Rightarrow l_S(g \a^k) &\leq \Lambda^k l_S(g) \hspace{35pt} \forall g\in G
			\end{align*}
			\begin{align*}
				\Rightarrow &\text{GR}_\mathcal{G}(\a, g) = \limsup_{k\to \infty}\sqrt[k]{l_S(g \a^k)}
				\leq \limsup_{k\to \infty}\sqrt[k]{\Lambda^k l_S(g)}
				= \Lambda \hspace{35pt} \forall g\in G\\
				\Rightarrow &\text{GR}_\mathcal{G}(\a) \leq \lambda_S
			\end{align*}
			This holds for all $S\in \mathcal{O}$. Thus $\text{GR}_\mathcal{G}(\a)\leq \inf_{S\in\mathcal{O}} \Lambda(S,\a S) = \lambda_\a$.

			\item[\textbf{C:}]
			
			Recall the definition of the displacement: $\lambda_\a := \inf_{S\in\mathcal{O}}\Lambda_R(S,\a S)$. Ideally we would prove inequality $\textbf{C}$ by finding a $G$-tree in $\mathcal{O}$ whose right stretching factor is exactly $\mu_R$. However, unless $\a$ is irreducible, this is not always possible. Thus we shall instead find a \emph{sequence} of $G$-trees whose right stretching factors \emph{tend} towards $\mu_R$.
			
			Our proof will apply \cite[Lemma 4.3]{francavigliamartino2015}, which states that the Lipschitz constant of an $\mathcal{O}$-map is equal to the right stretching factor between its endpoints. Recall that relative train tracks are straight $\mathcal{O}$-maps - that is to say, they have constant speed on edges. Since we only have finitely many edge orbits in $T$, it follows that $\text{Lip}(f)$ is realised by some edge - i.e. $\text{Lip}(f) = \max_{e\in T} \frac{l_{T}(f(e))}{l_{T}(e)}$.
			
			Let $H_r$ be a stratum in $T$. Recall that the lengths of the edges in $T$ were determined by an eigenvector corresponding to $\mu_r$. Since eigenvectors are only determined up to scalar multiplication, we are free to rescale the edge lengths in $H_r$ by a positive constant without affecting any of the relevant properties.
			
			For a positive number $N>0$, let $T_N\in \mathcal{O}$ denote the $G$-tree acquired from $G$ by rescaling the stratum $H_r$ in $T$ by $N^r$ for every $r$. $f$ induces a relative train track map on $T_N$, which we shall denote by $f_N$.
			
			Since we are now working with multiple trees, we should make some adjustments to our $r$-length notation for this portion of the proof. $L_r$ shall be used to denote the $r$-lengths of the original tree $T$. We introduce the notation $L_{r,N}$ for the $r$-lengths of the tree $T_N$.
			
			Let $e$ be an edge in the $r$th stratum of $T_N$. Then $l_{T_N}(e) = L_{r,N}(e) = N^r L_r(e)$.
			\begin{align*}
				\Rightarrow l_{T_N}(f_N(e)) &= \sum_{i=1}^{r} L_{i,N}(f_N(e)) = \sum_{i=1}^{r} N^i L_{i}(f_N(e))\\
				\Rightarrow \frac{l_{T_N}(f(e))}{l_{T_N}(e)} &= \frac{\sum_{i=1}^r N^i L_{i} (f(e))}{N^r L_r(e)}\\
				&= \frac{\sum_{i=1}^r N^{i-r} L_i (f(e))}{ L_r(e)}\\
				&\too \frac{L_r(f(e))}{L_r(e)} \text{ as } N\to \infty\\
			\end{align*}
			By Lemma \ref{lem:propertiesofRTTs} (i) \& (iv), $\displaystyle\frac{L_r(f(e))}{L_r(e)} = \frac{\mu_r L_r(e)}{L_r(e)} = \mu_r$
			\begin{align*}
				\Rightarrow \lim_{N\to \infty}(\text{Lip}(f_N)) &= \lim_{N\to \infty}\bigg(\max_{e\in T_N} \frac{l_{T_N}(f(e))}{l_{T_N}(e)}\bigg) = \mu_R
			\end{align*}
			Finally,
			\begin{align*}
				\lambda_\a = \inf_{S\in\mathcal{O}}\Lambda_R(S,\a S) \leq \lim_{N\to \infty} \Lambda_R(T_N,\a T_N) \stackrel{\text{by \cite[Lem 4.3]{francavigliamartino2015}}}{\leq} \lim_{N\to \infty}(\text{Lip}(f_N)) = \mu_R
			\end{align*}
			
		\end{itemize}
		
	\end{proof}

	\section{Appendix 1: A Bounding Function}
	
	Let $\a \in \text{Out}(G,\mathcal{G})$. As in the previous chapter, we consider a relative train track map $f:T\to \a T$ on a $G$-tree $T\in \mathcal{O}$. The purpose of this appendix is to find a function of $k$ which acts as an upper bound on $l_T(g \a^k)$. If $g$ is elliptic, then $l_T(g \a^k) = 0$, so we shall restrict our attention to hyperbolic elements.
	
	As before, $f$ gives a stratification of $T$. We begin by considering the effect $\a$ has on a single stratum:
	
	\begin{lem}\label{lem:k_equals_1}
		Let $g\in \text{Hyp}(\mathcal{G})$. Then, for all $r$,
		\begin{align*}
			L_r(g\a) \leq \sum_{i\geq r}A(r,i) L_i(g) \text{, where } A(r,i) = \max_{e\in H_i}\frac{L_r(f(e))}{l(e)}
		\end{align*}
	\end{lem}
	\begin{proof}
		Recall the notation we introduced earlier in the paper: When we write $L_r(g \a)$, we actually mean $L_r(f(\gamma_g))$, where $\gamma_g$ is an edge path serving as a fundamental domain of the axis of $g$.
		
		The filtration $\emptyset = T_0\subset \ldots \subset T_n = T$ determined by $f$ is $f$-invariant - therefore a point on the path $f(\gamma_g)$ can only lie in $H_r$ if it was mapped from $H_r$ itself or from a higher stratum. Thus we shall split $\gamma_g$ into pieces $\gamma_g\cap H_i$ for $i\geq r$ and consider the effect $f$ has on each piece.
		
		(Note that $\gamma_g\cap H_i$ may not be connected. When we write $L_r(\gamma_g\cap H_i)$, we mean the sum of the $r$-lengths of its component paths.)
		\begin{align*}
			L_r(g\a) = L_r(f(\gamma_g)) = \sum_{i\geq r} L_r(f(\gamma_g\cap H_i))
		\end{align*}
		Let $i\geq r$. Then
		\begin{align*}
			L_R(f(\gamma_g\cap H_i)) &= \sum_{\text{edge orbits }e\in H_i} n_e L_R(f(e)) &&\text{where $n_e = $ num. times $\gamma_g$ crosses $e$}\\
			&= \sum_{\text{edge orbits }e\in H_i} n_e \frac{L_R(f(e))}{l(e)}l(e)\\
			&\leq A(r,i) \sum_{\text{edge orbits }e\in H_i} n_e l_T(e)\\
			&= A(r,i) L_i(\gamma_g)\\
			&= A(r,i) L_i(g)
		\end{align*}
		Thus
		\begin{align*}
			L_r(g\a) \leq \sum_{i\geq r}A(r,i) L_i(g)
		\end{align*}
	\end{proof}
	
	We now iterate and consider the effect $\a^k$ has on a single stratum.
	
	\begin{lem}\label{lem:full_induction}
		Let $g\in G$. Then for all $k\geq 1$,
		\begin{align*}
			L_r(g\a^k) &\leq \sum_{(i_1,\ldots,i_k)\in I_k[r,m]}A(r,i_1) A(i_1, i_2) \ldots A(i_{k-1}, i_k) l_{i_k}(g)\\
			\text{where } I_k[r,m] &= \{(i_1,\ldots,i_k)\in \Z_k \mid r\leq i_1\leq \ldots \leq i_m \leq m\}\\
			\text{and } A(i,j) &= \max_{e\in H_j}\frac{L_i(f(e))}{l(e)}
		\end{align*}
	\end{lem}
	\begin{proof}
		We prove this by induction. The case of $k=1$ is given by Lemma \ref{lem:k_equals_1}.
		
		Assume that the statement holds for $k=n$. Then
		\begin{align*}
			L_r(g\a^{n+1}) &\leq \sum_{j\geq r}A(r,j) L_j(g \a^n) &&\text{(by Lemma \ref{lem:k_equals_1})}\\
			&\leq \sum_{j\geq r}A(r,j) \Bigg(\sum_{(i_1,\ldots,i_n)\in I[j,m,n]}A(j,i_1) A(i_1, i_2) \ldots A(i_{n-1}, i_n) L_{i_n}(g) \Bigg) &&\text{(by inductive hypothesis)}\\
			& = \sum_{j\geq r} \Bigg(\sum_{(i_1,\ldots,i_n)\in I[j,m,n]}A(r,j) A(j,i_1) A(i_1, i_2) \ldots A(i_{n-1}, i_n) L_{i_n}(g) \Bigg)\\
			& = \sum_{(j, i_1,\ldots,i_n)\in I[r,m,n+1]}A(r,j) A(j,i_1) A(i_1, i_2) \ldots A(i_{n-1}, i_{n}) L_{i_n}(g)\\
			& = \sum_{(i_1,\ldots,i_{n+1})\in I[r,m,n+1]}A(r,i_1) A(i_1, i_2) \ldots A(i_{n}, i_{n+1}) L_{i_{n+1}}(g) &&\text{(after relabelling indices)}
		\end{align*}
		Thus the statement holds for all $k\geq 1$.
		
	\end{proof}
	
	\begin{thm}\label{thm:polynomial_bound}
		Let $\a\in\text{Out}(G,\mathcal{G})$. Let $f:T\to \a T$ be a relative train track map on some $T\in\mathcal{O}(G,\mathcal{G})$. Let $\mu_R$ be the largest Perron-Frobenius eigenvalue of a stratum in $T$, and let $m$ be the total number of strata in $T$. Then there exists a polynomial $P(k)$ of degree at most $m-1$ such that
		\begin{center}
			$l_T(g \a^k) \leq P(k)\mu_R^k l_T(g)$
		\end{center}
		for all $g\in G$.
	\end{thm}
	
	\begin{proof}
		It follows from Lemma \ref{lem:full_induction} that for all $k$,
		\begin{align*}
			l_T(g\a^k) &= \sum_{r=1}^m L_R(g\a^k)\\
			&\leq \sum_{r=1}^m \Bigg(  \sum_{(i_1,\ldots,i_{k+1})\in I_{k+1}[r,m]}A(r,i_1) A(i_1, i_2) \ldots A(i_{k}, i_{k+1}) L_{i_{k+1}}(g) \Bigg)
		\end{align*}
		Thus we have an upper bound for $l_T(g\a^k)$. All that remains is to simplify it.
		
		First we will find the size of the set $I_k[r,m]$, which will tell us how quickly the number of terms in our sum grows as $k\to \infty$. We can think of $I_k[r,m]$ as a set of combinations with repetition, where we choose $k$ elements from the set $\{r,\ldots,m\}$ (which has size $m-r+1$). Therefore
		\begin{align*}
			| I_k[r,m] | &= C(m-r+1, k)\\
			&= \frac{(k+m-r)!}{k!(m-r)!}\\
			&= \frac{(k+m-r)(k+m-r-1)\ldots (k+1)}{(m-r)!}
		\end{align*}
		This is a polynomial of degree $m-r$ in $k$.
		
		Secondly, consider the coefficients of the sum: $A(r,i_1) A(i_1, i_2) \ldots A(i_{k-1}, i_k) l_{i_k}(g)$. If $i=j$, then $A(i,j) = \mu_i\leq \mu_R$; if $i\neq j$, then $A(i,j)$ can appear in each coefficient at most once. It follows that
		\begin{align*}
			A(r,i_1) A(i_1, i_2) \ldots A(i_{k-1}, i_k) \leq \mathcal{A}\mu_R^k
		\end{align*}
		where
		\begin{align*}
			\mathcal{A} = \prod_{i\neq j} A(i,j) 
		\end{align*}
		
		Thirdly, we observe that $l_{T,i}(g)\leq l_T(g)$ for all $i$. Thus, combining all of the above:
		\begin{align*}
			l_T(g\a^k) &\leq \sum_{r=1}^m \Bigg( \sum_{(i_1,\ldots,i_k)\in I_k[r,m]}A(r,i_1) A(i_1, i_2) \ldots A(i_{k-1}, i_k) l_{T,i_k}(g) \Bigg)\\
			&\leq \sum_{r=1}^m \Big(| I_k[r,m] | \mathcal{A}\mu_R^k l_T(g) \Big)\\
			&=P(k) \mu_R^k l_T(g) \hspace{35pt}
		\end{align*}
		where $P(k) = \mathcal{A} \sum_{r=1}^m | I_k[r,m] |$ is a polynomial of degree at most $m-1$.
		
	\end{proof}


	
	
	

	\section{Appendix 2: Irreducible Automorphisms}\label{sec:irreduciblecase}
	
	In the case when $\a\in \text{Out}(G,\mathcal{G})$ is irreducible, the proof that the relative growth rate and displacement are equal is simpler. This is because we can guarantee the existence of train track maps, not just relative train track maps. Furthermore, we can guarantee the existence of a point in $\mathcal{O}$ which realises the displacement, not just a sequence which tends towards it.

	\begin{defn}
		Let $\alpha\in \text{Out}(G,\mathcal{G})$, $T\in \mathcal{O}(G,\mathcal{G})$. We say that a topological representative $f:T\to \alpha T$ is a \emph{train track map} if every edge $e\in T$ is legal with respect to $f$.
	\end{defn}

	When $\a$ is an irreducible automorphism, the transition matrix of any topological representative $f:T\to \alpha T$ will be an irreducible matrix - thus there is only one stratum of $T$ (the whole tree), and only one PF-eigenvalue $\mu$. As with the more general case, the positive row eigenvector corresponding to $\mu$ decides the metric on $T$. This ensures that the length of every edge is scaled exactly by $\mu$, so $\forall e\in T$, $l(f(e)) = \mu l(e)$.

	\begin{lem}\label{lem:TTproperty}
		Let $f:T\to \a T$ be a train track map representing an irreducible automorphism $\a$. Take the metric on $T$ to be the one determined by the PF-eigenvector. Then $\exists g \in \text{Hyp}(\mathcal{G})$ such that $l_T(g\a^k) = \mu^k l_T(g)$ for all $k>0$.
	\end{lem}
	\begin{proof}
	This is a well-known property of train track maps, but it also follows from Lemma~\ref{lem:propertiesofRTTs} (v) since we can think of a train track map as a relative train track map with only a single stratum. In this specific case, legal and $r$-legal are equivalent conditions, and the $r$-length of a path is equal to its length. Thus the result follows.

	\end{proof}

	\begin{thm}
		Let $\mathcal{G}$ be a free factor system for a group $G$, let $E$ be any relative generating set for $G$, and let $\alpha \in \text{Out}(G,\mathcal{G})$ be irreducible. Then $\text{GR}_{\mathcal{G}}(\a, l_E) = \lambda_\a$.
	\end{thm}
	\begin{proof}
		$\text{Min}(\a)$ is equal to the train-track bundle $\text{TT}(\a)$ - the set of $T\in \mathcal{O}$ admitting train track representatives $f:T \to \a T$ with $\text{Lip}(f) = \Lambda_R(T,\a T)$ \cite[Thmm 8.19, Thm 6.11]{francavigliamartino2015}. In addition, since $\alpha$ is irreducible, $\text{Min}(\a)$ is non-empty \cite[Theorem 8.4]{francavigliamartino2015}.
		
		Thus we can guarantee the existence of a train track map $f:T\to \alpha T$ on some $T$ such that:
		
		\begin{align*}
			\mu \stackrel{f \text{ is train track}}{=} \text{Lip}(f) \stackrel{f\in \text{TT}(\a)}{=} \Lambda_R(T,\alpha T) \stackrel{\text{by defn of Min}(\a)}{=} \lambda_\a
		\end{align*}
		
		In addition, since $l_E\sim_{\mathcal{G}} l_T$, these length functions will produce the same growth rate.
		
		Therefore in order to prove that $\text{GR}_{\mathcal{G}}(\a, l_E) = \lambda_\a$ it suffices to prove that $\text{GR}_{\mathcal{G}}(\a, l_T) = \mu$. The proof of this follows from the definition of the right stretching factor $\Lambda_R$. Recall,
		\begin{align*}
			\Lambda_R(T,\alpha^k T) := \sup_{g\in \text{Hyp}(\mathcal{G})}\dfrac{l_T(g\alpha^k )}{l_T(g)}
		\end{align*}
		
		$\Rightarrow$ For all $g\in G$, $l_T(g \a^k)\leq \mu^k l_T(g)$.
		
		$\Rightarrow$ For all $g\in G$, $GR_\mathcal{G}(\a, g, l_T) = \limsup_{k\to \infty}\sqrt[k]{l_T(g \a^k)}\leq \limsup_{k\to \infty}\sqrt[k]{\mu^k l_T(g)} = \mu$.
		
		Additionally, by Lemma~\ref{lem:TTproperty}, there exists $h\in G$ such that $\mu^k l_T(h) = l_T(h \a^k)$
		
		$\Rightarrow GR_\mathcal{G}(\a, h, l_T) = \limsup_{k\to \infty}\sqrt[k]{l_T(h \a^k)}= \limsup_{k\to \infty}\sqrt[k]{\mu^k l_T(h)} = \mu$
		
		Thus $GR_\mathcal{G}(\a, l_T) := \sup_{g}GR_\mathcal{G}(\a, g, l_T) = \mu$, and we are done.
	\end{proof}


\end{document}